\documentclass[a4paper, 12pt]{amsart}

\usepackage{amsmath}
\usepackage{amssymb}
\usepackage{ascmac}
\usepackage{amsthm}

\makeatletter

\@addtoreset{equation}{section}
\makeatother 

\theoremstyle{plain}
\newtheorem{thm}{Theorem}[section]
\newtheorem*{thm*}{Theorem}
\newtheorem{prop}[thm]{Proposition}
\newtheorem{lem}[thm]{Lemma}

\theoremstyle{definition}

\theoremstyle{remark}
\newtheorem{rem}[thm]{Remark}

\renewcommand{\epsilon}{\varepsilon}

\newcommand{\Lip}{\mathrm{Lip}}

\newcommand{\rarrow}{\rightarrow}

\newcommand{\olarrow}{\overleftarrow}

\usepackage{latexsym}

\title[Heat flow and concentration of measure]{Heat flow and concentration of measure on\\ directed graphs with a lower Ricci curvature bound}

\author{Ryunosuke Ozawa}
\address{Department of Mathematics, National Defense Academy of Japan, 1-10-20 Hashirimizu, Yokosuka, 239-8686 Japan}
\email{rozawa@nda.ac.jp}

\author{Yohei Sakurai}
\address{Department of Mathematics, Saitama University, 255 Shimo-Okubo, Sakura-ku, Saitama-City, Saitama, 338-8570, Japan}
\email{ysakurai@rimath.saitama-u.ac.jp}

\author{Taiki Yamada}
\address{Interdisciplinary Faculty of Science and Engineering, Shimane University, 1060 Nishikawatsu-cho, Matsue, Shimane, 690-8504, Japan}
\email{taiki\_yamada@riko.shimane-u.ac.jp}

\date{February 24, 2022}

\begin{document}
\maketitle

\begin{abstract}
In a previous work \cite{OSY1}, the authors introduced a Lin-Lu-Yau type Ricci curvature for directed graphs referring to the formulation of the Chung Laplacian.
The aim of this note is to provide a von Renesse-Sturm type characterization of our lower Ricci curvature bound via a gradient estimate for the heat semigroup, and a transportation inequality along the heat flow.
As an application,
we will conclude a concentration of measure inequality for directed graphs of positive Ricci curvature.
\end{abstract}

\section{Introduction}

\subsection{Main results}

On a smooth Riemannian manifold,
von Renesse-Sturm \cite{RS} have characterized a lower Ricci curvature bound in terms of a gradient estimate for the heat semigroup, and a transportation inequality along the heat flow.
M\"unch-Wojciechowski \cite{MW} have studied its discrete analogue,
and produced a characterization of a lower bound of the Ricci curvature introduced by Lin-Lu-Yau \cite{LLY} on (undirected) graphs (see also \cite{O}).

In \cite{OSY1},
the authors have introduced a Lin-Lu-Yau type Ricci curvature for directed graphs inspired by the formulation of the Chung Laplacian (\cite{C1}, \cite{C2}).
In the present note,
we aim to extend the results of M\"unch-Wojciechowski \cite{MW} to our directed setting.
Let us introduce our main results (see Section \ref{sec:Preliminaries} for the precise meaning of the notations).
Let $(V,\mu)$ be a simple,
strongly connected,
finite weighted directed graph,
where $V$ is the vertex set,
and $\mu:V\times V\to [0,\infty)$ is the (non-symmetric) edge weight.
We denote by $d:V\times V\to [0,\infty)$ the (non-symmetric) graph distance function on $V$,
and by $W$ the Wasserstein distance function.
For $x,y\in V$ with $x\neq y$,
let $\kappa(x,y)$ stand for the Ricci curvature introduced in \cite{OSY1}.
For $f:V\to \mathbb{R}$,
let $\Lip f$ denote its Lipschitz constant.
We denote by $\mathcal{L}$ the Chung Laplacian,
and by $P_t$ the heat semigroup
\begin{equation}\label{eq:heat semigroup}
P_t:=e^{-t \mathcal{L}}.
\end{equation}
Furthermore,
let $p^{t}_{x}$ be the associated heat kernel measure at $x$.

In \cite{OSY1},
the authors have derived a representation formula for the Ricci curvature in terms of the Chung Laplacian (see Theorem \ref{thm:limit free formula} below).
With the help of such a formula,
we first prove the following characterization theorem in our directed setting:
\begin{thm}\label{thm:RS}
Let $(V,\mu)$ denote a simple,
strongly connected,
finite weighted directed graph.
For $K\in \mathbb{R}$,
the following are equivalent:
\begin{enumerate}\setlength{\itemsep}{+0.5mm}
\item $\inf_{x\neq y}\kappa(x,y)\geq K$; \label{enum:Curv}
\item for all $f:V\to \mathbb{R}$ and $t>0$,
\begin{equation}\label{eq:Grad}
\Lip \,P_t f\leq e^{-Kt}\,\Lip\, f;
\end{equation}\label{enum:Grad}
\item for all $x,y\in V$ and $t>0$,
\begin{equation}\label{eq:Transport}
W(p^{t}_x,p^{t}_{y})\leq e^{-Kt}d(x,y).
\end{equation}\label{enum:Transport}
\end{enumerate}
\end{thm}

Theorem \ref{thm:RS} has been obtained by M\"unch-Wojciechowski \cite{MW} in the undirected case (see \cite[Theorem 3.8]{MW}).
Notice that
in the undirected case,
they have proven it not only for finite graphs but also for infinite graphs.

Having Theorem \ref{thm:RS} at hand,
we will investigate the concentration of measure phenomena with respect to the Perron measure $\mathfrak{m}$.
Due to the lack of symmetry of the distance function,
we need to introduce the following notion:
For $x,y\in V$,
we set
\begin{equation*}
\mathcal{D}(x,y):=\max\{d(x,y),d(y,x)\},\quad \mathcal{D}_x:=\sup_{y\in \mathcal{N}_x}\mathcal{D}(x,y),
\end{equation*}
where $\mathcal{N}_x$ is the neighborhood of $x$.
Note that $\mathcal{D}_x\geq 1$ in general.
In the undirected case,
the distance function $d$ is symmetric,
and hence we possess $\mathcal{D}_x=1$.
Taking the asymmetry into account,
we conclude the following concentration inequality for directed graphs:
\begin{thm}\label{thm:concentration}
Let $(V,\mu)$ denote a simple,
strongly connected,
finite weighted directed graph.
For $K>0$ and $\Lambda \geq 1$,
we assume $\inf_{x\neq y}\kappa(x,y)\geq K$ and $\sup_{x\in V}\mathcal{D}_x\leq \Lambda$.
Then for every $1$-Lipschitz function $f:V\to \mathbb{R}$
we have
\begin{equation*}
\mathfrak{m}(\{f\geq \mathfrak{m}(f)+ r\})\leq e^{-\frac{Kr^2}{\Lambda^2}},
\end{equation*}
where $\mathfrak{m}(f)$ is the mean of $f$ defined as
\begin{equation*}
\mathfrak{m}(f):=\sum_{x\in V}f(x)\mathfrak{m}(x).
\end{equation*}
\end{thm}

In the undirected case,
Jost-M\"unch-Rose \cite{JMR} have shown Theorem \ref{thm:concentration} based on the method of \cite{S} (see \cite[Theorem 3.1]{JMR}).
Also,
Fathi-Shu \cite{FS} have (implicitly) shown a similar result for reversible Markov chains via functional inequalities (see \cite[Theorems 1.13 and 2.4]{FS}, and cf. \cite{ELL}).

\subsection{Organization}\label{sec:Organization}
In Section \ref{sec:Preliminaries},
we will review basics of directed graphs.
In Section \ref{sec:Heat flow},
we prove Theorem \ref{thm:RS}.
In Sections \ref{sec:Concentration} and \ref{sec:Functional inequality},
we prove Theorem \ref{thm:concentration} in two different ways.
In Section \ref{sec:Concentration},
we do it by following the argument of \cite{S} as in \cite{JMR}.
In Section \ref{sec:Functional inequality},
we examine functional inequalities such as transportation-information inequality and transportation-entropy inequality,
and apply them to another proof of Theorem \ref{thm:concentration} as in \cite{FS}.

\section{Preliminaries}\label{sec:Preliminaries}
We review basics on directed graphs.
We refer to \cite{OSY1}.

\subsection{Directed graphs}\label{sec:Directed graphs}
Let $(G,\mu)$ be a finite weighted directed graph,
namely,
$G=(V,E)$ is a finite directed graph,
and $\mu:V \times V\to [0,\infty)$ is a function such that $\mu(x,y)>0$ if and only if $x\rightarrow y$,
where $x\rightarrow y$ means $(x,y) \in E$.
The function $\mu$ is called the \textit{edge weight},
and we write $\mu(x,y)$ by $\mu_{xy}$.
Note that
$(G,\mu)$ is undirected if and only if $\mu_{xy}=\mu_{yx}$ for all $x,y\in V$,
and simple if and only if $\mu_{xx}=0$ for all $x\in V$.
It is also called \textit{unweighted} if $\mu_{xy}=1$ whenever $x\rightarrow y$.
The weighted directed graph can be written as $(V,\mu)$ since the full information of $E$ is included in $\mu$.

For $x \in V$, 
its \textit{outer neighborhood $N_{x}$}, \textit{inner one $\olarrow{N}_{x}$}, and \textit{neighborhood $\mathcal{N}_{x}$} are defined as
\begin{equation*}\label{eq:neighborhoods}
N_{x}:= \left\{y \in V \mid x\rightarrow y \right\},\quad \olarrow{N}_{x}:= \left\{y \in V \mid y \rightarrow x \right\},\quad \mathcal{N}_{x}:=N_{x} \cup \olarrow{N}_{x},
\end{equation*}
respectively.

A sequence $\left\{x_{i} \right\}_{i=0}^{l}$ of vertices is called a \textit{directed path} from $x$ to $y$ if $x_0=x,\,x_{l}=y$ and $x_{i}\rarrow x_{i+1}$ for all $i=0,\dots,l-1$,
where $l$ is called its length.
$(V,\mu)$ is said to be \textit{strongly connected} if
for any $x,y\in V$,
there is a directed path from $x$ to $y$.
For strongly connected $(V,\mu)$,
the (non-symmetric) distance function $d:V\times V\to [0,\infty)$ is defined as follows:
$d(x,y)$ is defined to be the minimum of the lengths of directed paths from $x$ to $y$.

For $f:V\to \mathbb{R}$,
its \textit{Lipschitz constant} is defined by
\begin{equation*}
\Lip f:=\sup_{x\neq y} \nabla_{xy} f,
\end{equation*}
where $\nabla_{xy}$ is the \textit{gradient operator} defined as
\begin{equation*}
\nabla_{xy}f:=\frac{f(y)-f(x)}{d(x,y)}.
\end{equation*}
For $L>0$,
$f$ is said to be \textit{$L$-Lipschitz} if
\begin{equation*}
\Lip f\leq L.
\end{equation*}
Let $\Lip_{L}(V)$ stand for the set of all $L$-Lipschitz functions.

\subsection{Chung Laplacian}\label{sec:Laplacians}
Hereafter,
let $(V,\mu)$ be a simple,
strongly connected,
finite weighted directed graph.
In this subsection,
we review the formulation of the Chung Laplacian introduced in \cite{C1}, \cite{C2}.
The \textit{transition probability kernel} $P:V\times V\to [0,1]$ is defined as
\begin{equation*}\label{eq:Markov kernel}
P(x,y):=\frac{\mu_{xy}}{\mu(x)},
\end{equation*}
where
\begin{equation*}
\mu(x):=\sum_{y\in V}\mu_{xy}.
\end{equation*}
Since $(V,\mu)$ is finite and strongly connected,
the Perron-Frobenius theorem ensures that
there is a unique (up to scaling) positive function $m:V\to (0,\infty)$ such that
\begin{equation}\label{eq:Perron Frobenius}
m(x)=\sum_{y\in V}m(y)P(y,x).
\end{equation}
A probability measure $\mathfrak{m}:V\to (0,1]$ satisfying (\ref{eq:Perron Frobenius}) is called the \textit{Perron measure} (or the \textit{stationary probability measure}).

Let $\mathfrak{m}$ be the Perron measure.
For a non-empty subset $\Omega \subset V$,
we set
\begin{equation*}\label{eq:measure}
\mathfrak{m}(\Omega):=\sum_{x\in \Omega}\mathfrak{m}(x).
\end{equation*}
The \textit{reverse transition probability kernel} $\olarrow{P}:V\times V\to [0,1]$,
and the \textit{mean transition probability kernel} $\mathcal{P}:V\times V\to [0,1]$ are defined as
\begin{equation*}\label{eq:perron vector}
\olarrow{P}(x,y):=\frac{\mathfrak{m}(y)}{\mathfrak{m}(x)}P(y,x),\quad \mathcal{P}:=\frac{1}{2}(P+\olarrow{P}).
\end{equation*}
Let $\mathcal{F}$ be the set of all functions on $V$.
The \textit{Chung Laplacian} $\mathcal{L}:\mathcal{F}\to \mathcal{F}$ is given by
\begin{equation*}\label{eq:Chung Laplacian}
\mathcal{L}f(x):=f(x)-\sum_{y\in V} \mathcal{P}(x,y)f(y).
\end{equation*}
We will also use the \textit{negative Laplacian} $\Delta:\mathcal{F}\to \mathcal{F}$ defined by
\begin{equation*}\label{eq:negative Laplacian}
\Delta:=-\mathcal{L}.
\end{equation*}

We define a function $\mathfrak{m}:V\times V \to [0,\infty)$ by
\begin{equation*}\label{eq:symmetric edge weight}
\mathfrak{m}(x,y):=\frac{1}{2}(\mathfrak{m}(x)P(x,y)+\mathfrak{m}(y)P(y,x))=\mathfrak{m}(x)\mathcal{P}(x,y).
\end{equation*}
We write $\mathfrak{m}(x,y)$ by $\mathfrak{m}_{xy}$.
The following properties hold:
(1) $\mathfrak{m}_{xy}=\mathfrak{m}_{yx}$;
(2) $\mathfrak{m}_{xy}>0$ if and only if $y\in \mathcal{N}_{x}$ (or equivalently, $x\in \mathcal{N}_{y}$);
(3) $\mathcal{P}(x,y)=\mathfrak{m}_{xy}/\mathfrak{m}(x)$.

The \textit{inner product} on $\mathcal{F}$ is defined by
\begin{equation*}
(f_{0},f_{1}):=\sum_{x\in V} f_{0}(x)f_{1}(x)\mathfrak{m}(x)=\mathfrak{m}(f_0 f_1).
\end{equation*}
Also,
the \textit{$\Gamma$-operator} is defined as
\begin{equation*}
\Gamma(f_{0},f_{1}):=\frac{1}{2}\left(\Delta(f_{0}\, f_{1})-f_{0}\, \Delta f_{1}-f_{1}\, \Delta f_{0} \right),\quad \Gamma(f):=\Gamma(f,f).
\end{equation*}
By direct computations, we see the following (cf. \cite[Proposition 7.1]{OSY1}):
\begin{equation*}\label{eq:gradient}
\Gamma(f_0,f_1)(x)=\frac{1}{2}\sum_{y\in V} (f_0(y)-f_0(x))(f_1(y)-f_1(x)) \mathcal{P}(x,y).
\end{equation*}
We possess the following integration by parts formula (see e.g., \cite[Proposition 2.4]{OSY1}):
\begin{prop}\label{prop:integration by parts}
Let $\Omega \subset V$ be a non-empty subset.
Then for all $f_{0},f_{1}:V\to \mathbb{R}$,
\begin{align*}
\sum_{x\in \Omega}\mathcal{L}f_{0}(x)f_{1}(x)\mathfrak{m}(x)
&=\frac{1}{2}\sum_{x,y\in \Omega}(f_{0}(y)-f_{0}(x))(f_{1}(y)-f_{1}(x))\mathfrak{m}_{xy}\\
&\quad -\sum_{x\in \Omega}\sum_{y\in V\setminus \Omega}(f_{0}(y)-f_{0}(x))f_{1}(x)\mathfrak{m}_{xy}.
\end{align*}
In particular,
\begin{equation*}
(\mathcal{L}f_{0},f_{1})=\mathfrak{m}(\Gamma(f_0,f_1))=(f_{0},\mathcal{L}f_{1}).
\end{equation*}
\end{prop}

Let $P_t$ be the heat semigroup defined as \eqref{eq:heat semigroup} such that
$P_t f$ solves
\begin{equation*}
\begin{cases}
		\partial_t u=\Delta u,\\
		u|_{t=0}=f.
		\end{cases}
\end{equation*}
From Proposition \ref{prop:integration by parts}
we can derive
\begin{equation}\label{eq:symHSG}
(P_t f_{0},f_{1})=(f_{0},P_t f_{1}).
\end{equation}
The \textit{heat kernel measure $p^{t}_{x}$} is defined as
\begin{equation*}
p^{t}_{x}:=\frac{\mathfrak{m}}{\mathfrak{m}(x)}P_t \delta_x 
\end{equation*}
for the Dirac measure $\delta_x$ at $x$.
The equality \eqref{eq:symHSG} yields
\begin{equation}\label{eq:HSGHKM}
P_t f(x)=\sum_{y\in V}p^{t}_{x}(y) f(y).
\end{equation}
Furthermore,
we see $P_t 1_{V}=1_{V}$ by \eqref{eq:symHSG},
and hence $p^{t}_x$ is a probability measure.

\subsection{Optimal transport theory}\label{sec:OT}

We next recall the basics of the optimal transport theory (cf. \cite{V}).
For two probability measures $\nu_{0},\nu_{1}$ on $V$,
a probability measure $\pi :V\times V\to [0,\infty)$ is called a \textit{coupling of $(\nu_{0},\nu_{1})$} if
\begin{equation*}
\sum_{y \in V}\pi(x,y) =\nu_{0}(x),\quad \sum_{x \in V}\pi(x,y) = \nu_{1}(y).
\end{equation*}
Let $\Pi(\nu_{0},\nu_{1})$ stand for the set of all couplings of $(\nu_{0},\nu_{1})$.
The \textit{Wasserstein distance} from $\nu_{0}$ to $\nu_{1}$ is defined as
\begin{equation}\label{eq:Wasserstein distance}
W(\nu_{0},\nu_{1}):=\inf_{\pi \in \Pi(\nu_{0},\nu_{1})} \sum_{x,y \in V}d(x,y)\pi(x,y),
\end{equation}
which is a (non-symmetric) distance function on the set of all probability measures on $V$.

The following Kantorovich-Rubinstein duality formula is well-known (cf. \cite[Theorem 5.10 and Particular Cases 5.4 and 5.16]{V}):
\begin{prop}\label{prop:Kantorovich duality}
For any two probability measures $\nu_{0}, \nu_{1}$ on $V$, we have
\begin{equation*}
W(\nu_{0},\nu_{1}) = \sup_{f\in \Lip_{1}(V)} \sum_{x\in V} f(x)\left(  \nu_{1}(x)-\nu_{0}(x)  \right).
\end{equation*}
\end{prop}

\begin{rem}
In the discussion of \cite[Particular Cases 5.4]{V},
the symmetry for cost functions is not required,
and hence Proposition \ref{prop:Kantorovich duality} holds in our directed setting.
\end{rem}

\subsection{Ricci curvature}\label{sec:Ricci curvature}
In this subsection,
we recall the formulation of the Ricci curvature introduced in \cite{OSY1}.
For $\epsilon \in [0,1]$,
and for $x,y\in V$ with $x\neq y$,
we set
\begin{equation*}\label{eq:pre Ricci curvature}
\kappa_{\epsilon}(x,y):=1-\frac{W(\nu^{\epsilon}_{x},\nu^{\epsilon}_{y})}{d(x,y)},
\end{equation*}
where $\nu^{\epsilon}_{x}:V\to [0,1]$ is a probability measure defined by
\begin{equation*}
\nu^{\epsilon}_{x}(z)=(1-\epsilon)\delta_{x}(z)+\epsilon \,\mathcal{P}(x,z).
\end{equation*}
The authors \cite{OSY1} have introduced the \textit{Ricci curvature} as follows (see \cite[Definition 3.6]{OSY1}):
\begin{equation*}
\kappa(x,y):=\lim_{\epsilon\to 0}\frac{\kappa_{\epsilon}(x,y)}{\epsilon},
\end{equation*}
which is well-defined (see \cite[Lemmas 3.2 and 3.4, and Definition 3.6]{OSY1}).
In the undirected case,
this is nothing but the Lin-Lu-Yau Ricci curvature in \cite{LLY}.

We have the following representation formula,
which has been established by M\"unch-Wojciechowski \cite{MW} in the undirected case (see \cite[Theorem 3.10]{OSY1} and \cite[Theorem 2.1]{MW}):
\begin{thm}[\cite{MW}, \cite{OSY1}]\label{thm:limit free formula}
\begin{equation*}
\kappa(x,y)=\inf_{f\in \mathcal{F}_{xy}}  \nabla_{xy} \mathcal{L} f,
\end{equation*}
where
\begin{equation*}
\mathcal{F}_{xy}:=\{f\in \Lip_{1}(V)  \mid \nabla_{xy}f=1\}.
\end{equation*}
\end{thm}

\section{Heat flow}\label{sec:Heat flow}
In this section,
we give a proof of Theorem \ref{thm:RS}.

\subsection{Curvature bound and gradient estimate}\label{sec:CurvatureGradient}
In this subsection,
we show the equivalence of (\ref{enum:Curv}) and (\ref{enum:Grad}) in Theorem \ref{thm:RS}.
\begin{prop}\label{prop:RS1}
For $K\in \mathbb{R}$,
the following are equivalent:
\begin{enumerate}\setlength{\itemsep}{+0.7mm}
\item $\inf_{x\neq y}\kappa(x,y)\geq K$; \label{enum:Curv1}
\item for all $f:V\to \mathbb{R}$ and $t>0$, we have \eqref{eq:Grad}.\label{enum:Grad1}
\end{enumerate}
\end{prop}
\begin{proof}
We begin with the implication from (\ref{enum:Curv1}) to (\ref{enum:Grad1}).
For a fixed $x,y\in V$,
and $f:V\to \mathbb{R}$,
it is enough to prove
\begin{equation*}
\nabla_{xy}P_t f\leq e^{-Kt} \Lip f
\end{equation*}
for every $t>0$.
If $f$ is constant,
then this immediately follows from $P_t 1_{V}=1_{V}$.
Thus we may assume that $f$ is non-constant;
in particular,
$\Lip f>0$.
We will prove it by contradiction.
Suppose that there exists $t>0$ such that
\begin{equation*}
e^{Kt}\nabla_{xy}P_t f> \Lip f.
\end{equation*}
We here notice that
\begin{equation*}
e^{Kt}\nabla_{xy}P_t f|_{t=0}=\nabla_{xy}f\leq  \Lip f.
\end{equation*}
In this case,
there is $t_0>0$ such that
\begin{equation*}
e^{Kt_0}\nabla_{xy}P_{t_0} f> \Lip f,\quad \partial_t (e^{Kt}\nabla_{xy}P_t f)|_{t=t_{0}}>0.
\end{equation*}
Due to Theorem \ref{thm:limit free formula},
\begin{equation*}
0<\frac{e^{-Kt_{0}}}{\nabla_{xy}P_{t_{0}} f}\partial_t (e^{Kt}\nabla_{xy}P_t f)|_{t=t_{0}}=\left( K-\nabla_{xy}\mathcal{L}\frac{P_{t_{0}} f}{\nabla_{xy}P_{t_{0}} f}    \right)\leq \left( K-\kappa(x,y)    \right),
\end{equation*}
and hence $\kappa(x,y)<K$.
This contradicts with the curvature bound.

We next show the opposite direction.
By \eqref{eq:Grad},
for any $f\in \mathcal{F}_{xy}$,
\begin{equation*}
\nabla_{xy}\mathcal{L}f=-\partial_t \nabla_{xy}P_tf|_{t=0}=\lim_{t\to 0}\frac{1}{t}(\nabla_{xy}f-\nabla_{xy}P_t f)\geq \lim_{t\to 0}\frac{1}{t}(1-e^{-Kt})=K,
\end{equation*}
here we used $\Lip f=1$.
Theorem \ref{thm:limit free formula} leads us to the desired conclusion.
\end{proof}

\subsection{Gradient estimate and transportation inequality}\label{sec:GradientTransport}
Next,
we prove the equivalence of (\ref{enum:Grad}) and (\ref{enum:Transport}),
and conclude Theorem \ref{thm:RS}.
\begin{prop}\label{prop:RS2}
For $K\in \mathbb{R}$,
the following are equivalent:
\begin{enumerate}\setlength{\itemsep}{+0.7mm}
\item For all $f:V\to \mathbb{R}$ and $t>0$, we have \eqref{eq:Grad};\label{enum:Grad2}
\item for all $x,y\in V$ and $t>0$, we have \eqref{eq:Transport}.\label{enum:Transport2}
\end{enumerate}
\end{prop}
\begin{proof}
We start with the implication from (\ref{enum:Grad2}) to (\ref{enum:Transport2}).
By \eqref{eq:HSGHKM} and \eqref{eq:Grad},
it holds that
\begin{align*}
\sum_{z\in V}f(z)(p^{t}_y(z)-p^{t}_x(z))&=P_t f(y)-P_t f(x)\\
&\leq \Lip P_t f \, d(x,y)\leq e^{-Kt}\Lip f \,d(x,y)\leq e^{-Kt} \,d(x,y)
\end{align*}
for every $f\in \Lip_1(V)$.
With the help of Proposition \ref{prop:Kantorovich duality},
we arrive at \eqref{eq:Transport}.

We now consider the opposite one.
We fix $f:V\to \mathbb{R}$.
If $f$ is constant,
then the desired assertion is trivial by $P_t1_V=1_V$.
If $f$ is non-constant,
then $\Lip f$ is positive,
and hence we can define $g\in \Lip_1(V)$ by
\begin{equation*}
g:=\frac{f}{\Lip f}.
\end{equation*}
Therefore,
\eqref{eq:HSGHKM}, Proposition \ref{prop:Kantorovich duality} and \eqref{eq:Transport} imply
\begin{align*}
P_t f(y)-P_t f(x)&=\Lip f (P_t g(y)-P_t g(x))=\Lip f \sum_{z\in V}g(z)(p^{t}_y(z)-p^{t}_x(z))\\
                       &\leq \Lip f \,W(p^{t}_x,p^{t}_y)\leq e^{-Kt}\Lip f d(x,y).
\end{align*}
By dividing the both sides by $d(x,y)$,
we complete the proof.
\end{proof}

We are now in a position to conclude Theorem \ref{thm:RS}.

\begin{proof}[Proof of Theorem \ref{thm:RS}]
Theorem \ref{thm:RS} is a direct consequence of Propositions \ref{prop:RS1} and \ref{prop:RS2}.
\end{proof}

\subsection{Characterization via heat flow}
In the undirected case,
M\"unch-Wojciechowski \cite{MW} have formulated not only the characterization of lower Ricci curvature bound but also that of Ricci curvature itself (see \cite[Theorem 5.8]{MW}).
In our directed setting,
we also have the following characterization of M\"unch-Wojciechowski type:
\begin{thm}\label{thm:MWcharacterization}
\begin{equation*}
\kappa(x,y)=\lim_{t\to 0}\frac{1}{t}\left( 1-\frac{W(p^{t}_{x},p^{t}_{y})}{d(x,y)}  \right).
\end{equation*}
\end{thm}

We can prove Theorem \ref{thm:MWcharacterization} by the same argument as in the proof of \cite[Theorem 5.8]{MW}.
We omit the proof.

\section{Concentration of measure}\label{sec:Concentration}
In this section,
we will prove Theorem \ref{thm:concentration} along the line of the proof of \cite[Theorem 3.1]{JMR}.

\subsection{Laplace functional estimates}\label{sec:Laplace functional}
For $\lambda \geq 0$,
the \textit{Laplace functional (or moment generating functional)} is defined as follows (see e.g., \cite[Subsection 1.6]{L}, \cite[Subsection 2.1]{BLM}):
\begin{equation*}
E(\lambda):=\sup_{f\in \Lip_1(V)}\mathfrak{m}(e^{\lambda f}),
\end{equation*}
where the supremum is taken over all $f\in \Lip_1(V)$ with $\mathfrak{m}(f)=0$.
In this subsection,
we give an upper bound of the Laplace functional under the same setting as in Theorem \ref{thm:concentration}.
To do so,
we prepare the following lemma (cf. \cite{S}):
\begin{lem}\label{lem:prepare}
For all $\lambda \geq 0$ and $f:V\to \mathbb{R}$ we have
\begin{equation*}\label{eq:prepare1}
\mathfrak{m}(\Gamma(f,e^{\lambda f}))\leq \lambda(e^{\lambda f},\Gamma(f)).
\end{equation*}
\end{lem}
\begin{proof}
We see
\begin{align*}
\mathfrak{m}(\Gamma(f,e^{\lambda f}))&=\frac{1}{2}\sum_{x,y\in V}(f(y)-f(x))(e^{\lambda f(y)}-e^{\lambda f(x)})\mathfrak{m}_{xy}\\
&=\sum_{f(y)>f(x)}(f(y)-f(x))(e^{\lambda f(y)}-e^{\lambda f(x)})\mathfrak{m}_{xy}.
\end{align*}
We now recall the following elementary inequality (see e.g., \cite[Corollary 5.8]{L}):
For all $s>t$,
\begin{equation*}
\frac{e^s-e^t}{s-t}\leq \frac{e^s+e^t}{2}.
\end{equation*}
If $f(y)>f(x)$,
then
\begin{equation*}
\frac{e^{\lambda f(y)}-e^{\lambda f(x)}}{f(y)-f(x)}\leq  \lambda \,\frac{e^{\lambda f(y)}+e^{\lambda f(x)}}{2}.
\end{equation*}
It follows that
\begin{align*}
\mathfrak{m}(\Gamma(f,e^{f}))&\leq \frac{\lambda}{2} \sum_{f(y)>f(x)}\left(  e^{\lambda f(y)}+e^{\lambda f(x)} \right)\,(f(y)-f(x))^2\mathfrak{m}_{xy}\\
                           &= \frac{\lambda}{2} \sum_{x,y\in V}e^{\lambda f(y)}\,(f(y)-f(x))^2\mathfrak{m}_{xy}=\lambda(e^{\lambda f},\Gamma(f)).
\end{align*}
This proves the lemma.
\end{proof}

We now state the desired assertion:
\begin{prop}\label{prop:Laplace estimate}
For $K>0$ and $\Lambda \geq 1$,
we assume $\inf_{x\neq y}\kappa(x,y)\geq K$ and $\sup_{x\in V}\mathcal{D}_x\leq \Lambda$.
Then for every $\lambda \geq 0$,
\begin{equation*}
E(\lambda)\leq e^{\frac{\lambda^2 \Lambda^2}{4K}}.
\end{equation*}
\end{prop}
\begin{proof}
Let $f\in \Lip_{1}(V)$ with $\mathfrak{m}(f)=0$.
Using Proposition \ref{prop:integration by parts} and Lemma \ref{lem:prepare},
we obtain
\begin{equation*}
\partial_t \, \mathfrak{m}(e^{\lambda P_t f})=-\lambda\,(\mathcal{L}P_t f,e^{\lambda P_t f})=-\lambda \, \mathfrak{m}(\Gamma(P_t f,e^{\lambda P_t f}))\geq -\lambda^2 (e^{\lambda P_t f},\Gamma(P_t f)).
\end{equation*}
By virtue of Theorem \ref{thm:RS},
\begin{align*}
\Gamma(P_t f)&=\frac{1}{2}\sum_{x,y\in V} (P_t f(y)-P_t f(x))^2\mathfrak{m}_{xy}\\
&\leq \frac{1}{2}(\Lip P_t f)^2 \sum_{x,y\in V}\mathcal{D}(x,y)^2 \mathfrak{m}_{xy}\leq \frac{1}{2}\Lambda^2 \,e^{-2Kt}(\Lip f)^2\leq \frac{1}{2}\Lambda^2 \,e^{-2Kt}.
\end{align*}
By combining the above inequalities,
\begin{equation*}
\partial_t \, \mathfrak{m}(e^{\lambda P_t f})\geq -\frac{\lambda^2 \Lambda^2}{2} e^{-2Kt}\mathfrak{m}(e^{\lambda P_t f}),
\end{equation*}
and hence
\begin{equation*}
\log \mathfrak{m}(e^{\lambda P_t f})-\log \mathfrak{m}(e^{\lambda f}) \geq \frac{\lambda^2 \Lambda^2}{4K}\left(e^{-2Kt}-1\right).
\end{equation*}
By letting $t\to \infty$ we arrive at
\begin{equation*}
\mathfrak{m}(e^{\lambda f})\leq e^{\frac{\lambda^2 \Lambda^2}{4K}}.
\end{equation*}
Here we used $\mathfrak{m}(e^{\lambda P_t f})\to 1$ as $t\to \infty$,
which is a consequence of $\mathfrak{m}(f)=0$ and the fact that $\mathfrak{m}$ is a probability measure.
This completes the proof.
\end{proof}

\subsection{Concentration inequalities}\label{sec:Concentration inequalities}

Let us recall the following Chernoff bounding method (see e.g., \cite[Proposition 1.14]{L}, \cite[Subsection 2.1]{BLM}):
\begin{prop}\label{prop:Chernoff}
Let $c>0$.
If 
\begin{equation*}
E(\lambda)\leq e^{\frac{\lambda^2}{2c}}
\end{equation*}
for all $\lambda \geq 0$,
then
\begin{equation*}
\mathfrak{m}(\{f\geq \mathfrak{m}(f)+ r\})\leq e^{-\frac{c r^2}{2}}.
\end{equation*}
\end{prop}

Now,
one can now derive Theorem \ref{thm:concentration}.

\begin{proof}[Proof of Theorem \ref{thm:concentration}]
Proposition \ref{prop:Laplace estimate} together with Proposition \ref{prop:Chernoff} with $c=2K/\Lambda^2$ implies Theorem \ref{thm:concentration}.
\end{proof}

\section{Functional inequalities}\label{sec:Functional inequality}
Here we discuss several functional inequalities,
and give another proof of Theorem \ref{thm:concentration}.

\subsection{Transportation-information inequality}\label{sec:Transport information inequality}

In this subsection,
we examine a transportation-information inequality.
We first show the following lemma (cf. \cite[Lemma 5.1]{FS}):
\begin{lem}\label{lem:element}
For $K>0$ and $\Lambda \geq 1$,
we assume $\inf_{x\neq y}\kappa(x,y)\geq K$ and $\sup_{x\in V}\mathcal{D}_x\leq \Lambda$.
Then for every probability density $\rho:V\to [0,\infty)$ $($i.e., $\mathfrak{m}(\rho)=1$$)$ we have
\begin{equation*}
W(\mathfrak{m},\rho\mathfrak{m})\leq \frac{\Lambda}{2K}\sum_{x,y\in V} \vert \rho(y)-\rho(x)\vert \mathfrak{m}_{xy}.
\end{equation*}
\end{lem}
\begin{proof}
Proposition \ref{prop:Kantorovich duality} can be written as
\begin{equation*}
W(\mathfrak{m},\rho \mathfrak{m}) = \sup_{f\in \Lip_1(V)}(f,\rho),
\end{equation*}
where the supremum is taken over all $f\in \Lip_1(V)$ with $\mathfrak{m}(f)=0$.
By Proposition \ref{prop:integration by parts},
\begin{equation*}
W(\mathfrak{m},\rho \mathfrak{m}) = -\sup_{f\in \Lip_1(V)} \int^{\infty}_{0} \,\frac{d}{dt} (P_t f,\rho)dt=\sup_{f\in \Lip_1(V)}\int_0^{\infty}\,\mathfrak{m}(\Gamma(P_t f,\rho))\,dt.
\end{equation*}
Let us fix $f\in \Lip_1(V)$ with $\mathfrak{m}(f)=0$.
Theorem \ref{thm:RS} tells us that
\begin{align*}
\int_0^{\infty}\,\mathfrak{m}(\Gamma(P_t f,\rho))\,dt
&=\frac{1}{2}\int_0^{\infty}\,  \sum_{x,y\in V}  (P_t f(y) - P_t f(x))(\rho(y) - \rho(x))\mathfrak{m}_{xy}\,dt \\
&\leq \frac{1}{2}\int_0^{\infty}{\Lip P_t f \,dt\sum_{x,y\in V} \mathcal{D}(x,y)  \vert \rho(y) - \rho(x)\vert \mathfrak{m}_{xy}} \\
&\leq \frac{\Lambda}{2K}\sum_{x,y\in V}  \vert \rho(y) - \rho(x)\vert \mathfrak{m}_{xy}.
\end{align*}
This proves the lemma.
\end{proof}

For a probability density $\rho:V\to [0,\infty)$,
the \textit{Fisher information} is defined by
\begin{equation*}
\mathcal{I}(\rho):=4 \mathfrak{m}(\Gamma (\sqrt{\rho}))=2\sum_{x,y\in V}\left(\sqrt{\rho(y)}-\sqrt{\rho(x)}\right)^2\mathfrak{m}_{xy}.
\end{equation*}
We now state our desired inequality (cf. \cite[Theorem 1.13]{FS}).
\begin{thm}\label{thm:transinfo}
For $K>0$ and $\Lambda \geq 1$,
we assume $\inf_{x\neq y}\kappa(x,y)\geq K$ and $\sup_{x\in V}\mathcal{D}_x\leq \Lambda$.
Then for every probability density $\rho:V\to [0,\infty)$ we have
\begin{equation*}
W(\mathfrak{m},\rho \mathfrak{m})^2\leq \frac{\Lambda^2}{2K^2}\mathcal{I}(\rho)\left( 1-\frac{1}{8}\mathcal{I}(\rho)  \right)\leq \frac{\Lambda^2}{2K^2}\mathcal{I}(\rho).
\end{equation*}
\end{thm}
\begin{proof}
Fathi-Shu \cite{FS} have proved a similar result for reversible Markov chains (see \cite[Theorem 1.13]{FS}).
We will prove it along the line of their argument.
Since $\rho$ is a probability density,
we see
\begin{align*}
\sum_{x,y\in V}\left(\sqrt{\rho(y)} + \sqrt{\rho(x)}\right)^2\mathfrak{m}_{xy}&=\sum_{x,y\in V}\left(2\rho(y)+2\rho(x)-\left(\sqrt{\rho(y)} - \sqrt{\rho(x)}\right)^2\right)\mathfrak{m}_{xy}\\
&= 4-\frac{1}{2}\mathcal{I}(\rho).
\end{align*}
From Lemma \ref{lem:element} we deduce
\begin{align*}
W(\mathfrak{m},\rho \mathfrak{m})&\leq \frac{\Lambda}{2K}\sum_{x,y\in V}  \vert \rho(y) - \rho(x)\vert \mathfrak{m}_{xy}\\
&= \frac{\Lambda}{2K}\sum_{x,y\in V}  \left\vert \sqrt{\rho(y)} - \sqrt{\rho(x)} \right\vert\left(\sqrt{\rho(y)} + \sqrt{\rho(x)}\right)\mathfrak{m}_{xy} \\
&\leq \frac{\Lambda}{2K}\sqrt{\mathcal{I}(\rho)}\sqrt{\frac{1}{2}\underset{x,y\in V}{\sum} \hspace{1mm} \left(\sqrt{\rho(y)} + \sqrt{\rho(x)}\right)^2\mathfrak{m}_{xy}}
\leq \frac{\Lambda}{2K}\sqrt{\mathcal{I}(\rho)}\sqrt{2-\frac{1}{4}\mathcal{I}(\rho)}.
\end{align*}
This completes the proof.
\end{proof}

\subsection{Transportation-entropy inequality}\label{sec:Transport entropy inequality}
We next investigate a transportation-entropy inequality.
For a probability density $\rho:V\to [0,\infty)$,
its \textit{relative entropy} is defined by
\begin{equation*}
\mathcal{E}(\rho):=\mathfrak{m}(\rho \log \rho).
\end{equation*}
We notice the following characterization (see e.g., \cite[(5.13)]{L}):
\begin{equation}\label{eq:entropyequiv}
\mathcal{E}(\rho)=\sup_{g}(g,\rho),
\end{equation}
where the supremum is taken over all $g:V\to \mathbb{R}$ with $\mathfrak{m}(e^{g})\leq 1$.

We verify the following Bobkov-G\"otze type criterion due to the lack of symmetry of the distance function (cf. \cite[Theorem 1.3]{BG}, \cite[Proposition 6.1]{L}):
\begin{lem}\label{lem:transportLap}
For $c>0$,
the following are equivalent:
\begin{enumerate}\setlength{\itemsep}{+0.7mm}
\item For every probability density $\rho:V\to [0,\infty)$ we have
\begin{equation}\label{eq:transportLap1}
W(\mathfrak{m},\rho \mathfrak{m})^2\leq \frac{2}{c}\mathcal{E}(\rho);
\end{equation}\label{enum:transportLap1}
\item for all $\lambda \geq 0$,
\begin{equation}\label{eq:transportLap2}
E(\lambda)\leq e^{\frac{\lambda^2}{2c}}.
\end{equation}\label{enum:transportLap2}
\end{enumerate}
\end{lem}
\begin{proof}
Let us show the implication from (\ref{enum:transportLap1}) to (\ref{enum:transportLap2}).
Fix $f\in \Lip_1(V)$ with $\mathfrak{m}(f)=0$,
and set
\begin{equation}\label{eq:gf}
g_f:=\lambda f-\frac{\lambda^2}{2c},\quad \rho_f:=\frac{e^{g_f}}{\mathfrak{m}(e^{g_f})}.
\end{equation}
By Proposition \ref{prop:Kantorovich duality},
\begin{equation*}
(f,\rho_f)\leq W(\mathfrak{m},\rho_f \mathfrak{m})\leq \sqrt{\frac{2}{c}\mathcal{E}(\rho_f)}\leq \frac{\lambda}{2c}+\frac{1}{\lambda}\mathcal{E}(\rho_f),
\end{equation*}
and hence $(g_f,\rho_f)\leq \mathcal{E}(\rho_f)$.
On the other hand,
straightforward computations imply
\begin{equation*}
\mathcal{E}(\rho_f)=(g_f,\rho_f)-\log \mathfrak{m}(e^{g_f}).
\end{equation*}
Therefore, $\mathfrak{m}(e^{g_f})\leq 1$,
which is equivalent to 
\begin{equation*}
\mathfrak{m}(e^{\lambda f}) \leq e^{\frac{\lambda^2}{2c}}.
\end{equation*}
We have shown the desired estimate.

We prove the opposite one.
Fix a probability density $\rho$.
Proposition \ref{prop:Kantorovich duality} can be written as
\begin{equation}\label{eq:opposite1}
W(\mathfrak{m},\rho \mathfrak{m}) = \sup_{f\in \Lip_1(V)}(f,\rho),
\end{equation}
where the supremum is taken over all $f\in \Lip_1(V)$ with $\mathfrak{m}(f)=0$.
We also fix such a $f$,
and define $g_f$ as \eqref{eq:gf}.
From \eqref{eq:transportLap2} we derive $\mathfrak{m}(e^{g_f})\leq 1$.
In view of \eqref{eq:entropyequiv},
\begin{equation*}
\lambda (f,\rho)-\frac{\lambda^2}{2c}=(g_f,\rho)\leq \mathcal{E}(\rho)
\end{equation*}
By letting $\lambda \to \sqrt{2c \mathcal{E}(\rho)}$,
we arrive at
\begin{equation}\label{eq:opposite2}
(f,\rho)\leq \sqrt{\frac{2}{c}\mathcal{E}(\rho)}.
\end{equation}
Combining \eqref{eq:opposite1} and \eqref{eq:opposite2},
we obtain \eqref{eq:transportLap1}.
We complete the proof.
\end{proof}

Proposition \ref{prop:Laplace estimate} together with Lemma \ref{lem:transportLap} yields the following transportation-entropy inequality:
\begin{thm}
For $K>0$ and $\Lambda \geq 1$,
we assume $\inf_{x\neq y}\kappa(x,y)\geq K$ and $\sup_{x\in V}\mathcal{D}_x\leq \Lambda$.
Then for every probability density $\rho:V\to [0,\infty)$
we have
\begin{equation*}
W(\mathfrak{m},\rho \mathfrak{m})^2\leq \frac{2\Lambda^2}{K}\mathcal{E}(\rho).
\end{equation*}
\end{thm}

\subsection{Relation between functional inequalities}
We finally mention the relation between the transportation-information inequality and the transportation-entropy inequality.
In order to do so,
we prepare the following (cf. \cite[Lemma 2.3]{FS}):
\begin{lem}\label{lem:prepare2}
For every $f:V\to \mathbb{R}$ we have
\begin{equation*}\label{eq:prepare1}
\mathfrak{m}(\Gamma(e^{f}))\leq (e^{2 f},\Gamma(f)).
\end{equation*}
\end{lem}
\begin{proof}
It holds that
\begin{equation*}
\mathfrak{m}(\Gamma(e^{f}))=\frac{1}{2}\sum_{x,y\in V}(e^{f(y)}-e^{f(x)})^2\,\mathfrak{m}_{xy}=\sum_{f(y)>f(x)}(e^{f(y)}-e^{f(x)})^2\,\mathfrak{m}_{xy}.
\end{equation*}
If $f(y)>f(x)$,
then we see
\begin{equation*}
\frac{e^{f(y)}-e^{f(x)}}{f(y)-f(x)}\leq \frac{e^{f(y)}+e^{f(x)}}{2}.
\end{equation*}
We also notice that
\begin{equation*}
\left(\frac{e^{f(y)}+e^{f(x)}}{2}\right)^2=\frac{e^{2f(y)}+e^{2f(x)}}{2}-\left(\frac{e^{f(y)}-e^{f(x)}}{2}\right)^2\leq \frac{e^{2f(y)}+e^{2f(x)}}{2}.
\end{equation*}
It follows that
\begin{align*}
\mathfrak{m}(\Gamma(e^{f}))&\leq \sum_{f(y)>f(x)}\left(\frac{e^{2f(y)}+e^{2f(x)}}{2}\right)\,(f(y)-f(x))^2\mathfrak{m}_{xy}\\
                           &=\frac{1}{2} \sum_{x,y\in V}e^{2 f(y)}\,(f(y)-f(x))^2\mathfrak{m}_{xy}=(e^{2 f},\Gamma(f)).
\end{align*}
This proves the lemma.
\end{proof}

We possess the following relation (cf. \cite[Theorem 2.4]{FS}, \cite[Theorem 2.1]{GLWW}):
\begin{prop}\label{prop:relation}
For $\Lambda \geq 1$,
we assume $\sup_{x\in V}\mathcal{D}_x\leq \Lambda$.
Let $c>0$,
and let $\rho:V\to [0,\infty)$ be a probability density.
If
\begin{equation}\label{eq:assumption}
W(\mathfrak{m},\rho \mathfrak{m})^2\leq \frac{1}{c^2}\mathcal{I}(\rho),
\end{equation}
then
\begin{equation*}
W(\mathfrak{m},\rho \mathfrak{m})^2\leq \frac{\sqrt{2}\Lambda}{c}\mathcal{E}(\rho).
\end{equation*}
\end{prop}
\begin{proof}
Fix $f\in \Lip_1(V)$ with $\mathfrak{m}(f)=0$,
and set
\begin{equation*}
\rho_{f,\lambda}:=\frac{e^{\lambda f}}{\mathfrak{m}(e^{\lambda f})}.
\end{equation*}
From Proposition \ref{prop:Kantorovich duality}, \eqref{eq:assumption}, and Lemma \ref{lem:prepare2},
it follows that
\begin{align*}
\frac{d}{d\lambda}\log \mathfrak{m}(e^{\lambda f})&=(f,\rho_{f,\lambda})
\leq W(\mathfrak{m},\rho_{f,\lambda} \mathfrak{m})\leq \sqrt{\frac{4}{c^2} \frac{1}{\mathfrak{m}(e^{\lambda f})}\mathfrak{m}(\Gamma(e^{\frac{\lambda f}{2}}))  }\\
&\leq \sqrt{\frac{4}{c^2} \frac{1}{\mathfrak{m}(e^{\lambda f})} \left( e^{\lambda f},\Gamma \left(\frac{\lambda f}{2} \right)\right) }= \sqrt{\frac{\lambda^2}{c^2} \frac{1}{\mathfrak{m}(e^{\lambda f})} \left( e^{\lambda f},\Gamma \left(f\right)\right) }.
\end{align*}
Now,
we have
\begin{equation*}
2\Gamma(f)(x)=\sum_{y\in V}(f(y)-f(x))^2\mathcal{P}(x,y)\leq \sum_{y\in V}\mathcal{D}(x,y)^2\mathcal{P}(x,y)\leq \Lambda^2,
\end{equation*}
and hence
\begin{equation*}
\frac{d}{d\lambda}\log \mathfrak{m}(e^{\lambda f})\leq \frac{\Lambda \lambda}{\sqrt{2}c}.
\end{equation*}
Integrating the both sides,
we obtain
\begin{equation*}
E(\lambda) \leq e^{\frac{\Lambda \lambda^2}{2\sqrt{2}c}}.
\end{equation*}
Thanks to Lemma \ref{lem:transportLap},
we arrive at the desired inequality.
\end{proof}

We are now in a position to provide another proof of Theorem \ref{thm:concentration}.

\begin{proof}[Another proof of Theorem \ref{thm:concentration}]
Theorem \ref{thm:transinfo} together with Lemma \ref{lem:transportLap} and Proposition \ref{prop:relation} with $c=\sqrt{2}K/\Lambda$ implies the same conclusion as in Proposition \ref{prop:Laplace estimate}.
Thus, we conclude Theorem \ref{thm:concentration} due to Proposition \ref{prop:Chernoff}.
\end{proof}

\subsection*{{\rm Acknowledgements}}
The authors are grateful to the anonymous referees for valuable comments.
The first named author was supported in part by JSPS KAKENHI (19K14532).
The first and second named authors were supported in part by JSPS Grant-in-Aid for Scientific Research on Innovative Areas ``Discrete Geometric Analysis for Materials Design" (17H06460).
The third named author was supported in part by JSPS KAKENHI (19K23411).


\end{document}